\documentclass{article}
\usepackage{amsmath,amssymb,amsthm} 
\usepackage{hyperref}

\usepackage{pgfplots}
\pgfplotsset{compat=1.15}
\usepackage{mathrsfs}
\usetikzlibrary{arrows}

\theoremstyle{theorem}
\newtheorem{theorem}{Theorem}
\newtheorem*{theorem*}{Theorem}
\newtheorem{proposition}{Proposition}
\newtheorem*{proposition*}{Proposition}
\newtheorem*{lemma*}{Lemma}
 
\theoremstyle{definition}

\newtheorem{remarkn}{Remark}
\newtheorem*{remark}{Remark}

\begin{document}

\title
{The surface tangent paradox and the difference vector quotient of a secant plane}

\author
{Paolo Roselli\footnote{
Università di Roma ``Tor Vergata'', 
Dipartimento di Matematica, 
via delle Ricerca Scientifica 1, 
00133 Roma, Italy. Email: roselli@mat.uniroma2.it}
}

\maketitle

\begin{abstract}
If a one-variable function is sufficiently smooth, then the limit position of secant lines its graph is a tangent line. By analogy, one would expect that the limit position of secant planes of a two-variable smooth function is a plane tangent to its graph. Amazingly, this is not necessarily true, even when the function is a simple polynomial. Despite this paradox, we show that some analogies with the one-variable case still hold in the multi-variable context, provided we use a particular vector product: the Clifford's geometric one.
\end{abstract}

\section{Introduction.}
The classical Schwarz surface area paradox\footnote{See \cite{Schwarz1890}, \cite{Peano1890}, and \cite{Zames1977}.} provides a counterexample concerning the definition of the area of a curved surface. It is a divergence phenomenon, that is  seldom presented, even at the end of an Advanced Calculus course\footnote{See \cite{FVB73} at pages 142--144.}.
Here, we show that a local version of this paradox deals with the very definition of the tangent to the graph of a smooth function. For simplicity, we only consider functions of one and two real variables. 

Let us first recall the link between lines secant and tangent the graph of a one-variable smooth function $g: I \subseteq \mathbb{R}\to\mathbb{R}$.  
If you consider the line passing through two distinct points $\big(a,g(a)\big)$, $\big(b,g(b)\big)\in I\times\mathbb{R}$ of the graph of~$g$,  then this secant line becomes the line tangent the graph of $g$ at point $\big(x_0,g(x_0)\big)$, as the distinct points~$a$ and~$b$ converge towards $x_0$, whatever is the way they do it, provided~$g$ is sufficiently smooth \big(a $C^1$-function, let's say\big). 
Analytically, this corresponds to the existence of the strong derivative\footnote{See \cite{Peano1892}, and \cite{Esser}.} of g at point $x_0$, which coincides with the classical derivative~$g'(x_0)$, provided the strong derivative exists (as it happens for a $C^1$-function). More precisely,
\begin{equation}
\label{eq:strong der one-var}
\lim_{
\arraycolsep=1.4pt\def\arraystretch{0.3}
\begin{array}{c}
\scriptstyle a,b\to x_0	\\
\scriptstyle a\ne b
\end{array}
}
\frac{g(b)-g(a)}{b-a}
=
g'(x_0)
\end{equation}
By analogy, one would expect that, for a two-variable smooth function, the plane tangent to its graph at a fixed point would be the limit position of the secant plane passing through three non-collinear points, as these three non-collinear points converge to that fixed point. 
Paradoxically, this is not the case.

\subsection{Counterexample.}
Let us consider as two-variable smooth function~$f$ the polynomial $f(\mathbf{v}) = \mathbf{v}\cdot \mathbf{v}= x^2+y^2$, when $\mathbf{v} = x \mathbf{e}_1 + y \mathbf{e}_2$, and $\{\mathbf{e}_1, \mathbf{e}_2\}$ is an orthonormal basis of the two-dimensional vector\footnote{Vectors, also called \textit{points}, will be indicated by \textbf{bold} Latin letters; real numbers, also called \textit{scalars}, will be denoted by lower-case Latin or Greek letters.} Euclidean space $\mathbb{E}_2$, with scalar product $\mathbf{u}\cdot\mathbf{w}$, with $\mathbf{u}, \mathbf{w}$ vectors of $\mathbb{E}_2$. So, $\mathbf{e}_1\cdot \mathbf{e}_2=0$, and $\mathbf{e}_1\cdot \mathbf{e}_1=\mathbf{e}_2\cdot \mathbf{e}_2=1$. If $\mathbf{a} = \mathbf{0}$, $\mathbf{b} = -\delta \mathbf{e}_1 + \eta \mathbf{e}_2$, 	and $\mathbf{c} = \delta \mathbf{e}_1 + \eta \mathbf{e}_2$, then the Cartesian equation of the plane passing through points $\big(\mathbf{a},f(\mathbf{a})\big)$, $\big(\mathbf{b},f(\mathbf{b})\big)$, and $\big(\mathbf{c},f(\mathbf{c})\big)$ can be represented, in the $(x,y,z)$-coordinates, by the relation
\begin{equation}
\label{eq:secant plane 1}
\eta z = (\delta^2 + \eta^2) y 
\end{equation}
As the limit of $\displaystyle \frac{\delta^2 + \eta^2}{\eta}$, as $\delta,\eta\to 0$, does not exist, then the secant plane~(\ref{eq:secant plane 1}) does not necessarily converge to the  plane $z=0$, which is tangent to  the graph of $f$ at point~$\mathbf{0}$. 

This counterexample illustrates the \textit{\textbf{tangent paradox}}:
a plane secant the graph of a smooth two-variable function~$f$ at points $\big(\mathbf{a},f(\mathbf{a})\big)$, $\big(\mathbf{b},f(\mathbf{b})\big)$, and $\big(\mathbf{c},f(\mathbf{c})\big)$ can assume limit positions that depend on the way the three non-collinear points $\mathbf{a}$, $\mathbf{b}$ and $\mathbf{c}$ converge to $\mathbf{x}_0$.
For example, if you consider $f$, $\mathbf{a}$, $\mathbf{b}$ and $\mathbf{c}$ as in the counterexample,
\begin{itemize}
	\item if $\eta=\delta \to 0$, then the limit plane would be $z=0$, the tangent plane;
	\item if $\eta=\delta^2 \to 0$, then the limit plane would be $z=y$, which is not the tangent plane;
	\item if $\eta=\delta^3 \to 0$, then the limit plane would be $y=0$, which is even orthogonal to the tangent plane!
\end{itemize}

So, contrary to the one-variable case, we cannot simply say that the tangent plane to a smooth two-variable function is the limit position of its secant planes.

\subsection{In search for analogies.} 
Despite such paradox, one can ask if any analogy with the one-variable case still holds for two-variable smooth functions. 

\noindent
In the $(x,z)$-plane, the Cartesian equation of the line passing through $\big(a,g(a)\big)$, $\big(b,g(b)\big)$ (with $a\ne b$) is
\begin{equation}
\label{eq:secant one-var}
z=g(a) +
\frac{g(b)-g(a)}{b-a}
(x-a)
\end{equation}
The existence of the strong derivative~(\ref{eq:strong der one-var}), and the continuity of~$g$ at~$x_0$ imply that, as $a,b\to x_0$, relation~(\ref{eq:secant one-var}) becomes relation
\begin{center}
$
z = g(x_0) + g'(x_0) (x-x_0)\ ,
$
\end{center}
which represents the line tangent the graph of $g$ at point $\big(x_0,g(x_0)\big)$. The relation representing the plane tangent the graph of the two-variable function~$f:\Omega\subseteq\mathbb{E}_2\to\mathbb{R}$ at point~$\big(\mathbf{v}_0,f(\mathbf{v}_0)\big)$ can be written in the $(\mathbf{v},z)$-space $\mathbb{E}_2\times\mathbb{R}$ as
\begin{center}
$
z = f(\mathbf{v}_0) + \nabla f(\mathbf{v}_0) \cdot (\mathbf{v}-\mathbf{v}_0)\ .
$
\end{center}
So, we wonder if, in analogy with~(\ref{eq:secant one-var}), there exists a  relation 
\begin{equation}
\label{eq:equiv two dim}
z=f(\mathbf{a}) + \mathbf{q} \cdot (\mathbf{v}-\mathbf{a})
\end{equation}
representing the plane passing through $\big(\mathbf{a},f(\mathbf{a})\big)$, $\big(\mathbf{b},f(\mathbf{b})\big)$, and $\big(\mathbf{c},f(\mathbf{c})\big)$, with vector~$\mathbf{q}=\mathbf{q}_{(f,\mathbf{a},\mathbf{b},\mathbf{c})}$ depending on $f$, $\mathbf{a}$, $\mathbf{b}$ and $\mathbf{c}$, as the classical difference scalar quotient~$q=q_{(g,a,b)}=\frac{g(b)-g(a)}{b-a}$ in~(\ref{eq:secant one-var}) depends on $g$, $a$, and $b$.

In this work we show\footnote{See Theorem~\ref{thm:main} at the end of this article.} that vector $\mathbf{q}_{(f,\mathbf{a},\mathbf{b},\mathbf{c})}$ can 
still be expressed as a difference quotient, provided the corresponding vector product is the Clifford's geometric vector product (\textit{\textbf{geometric product}}, for short). In order to show this, we will explain, in the next Section, what the geometric product is, at least between vectors in the two-dimensional Euclidean space $\mathbb{E}_2$. Here, we anticipate a consequence\footnote{See Theorem~\ref{thm:lin comb orth}.} of Theorem 1, that allows to express the vector quotient $\mathbf{q}$ as a fully symmetric linear combination of the outward  normals of the oriented triangle $T$ determined by points $\mathbf{a}$, $\mathbf{b}$, and $\mathbf{c}$:
\begin{equation}
\label{eq:lin comb orth}
\mathbf{q}
=
\mathbf{q}_{(f,\mathbf{a},\mathbf{b},\mathbf{c})}
=
\frac{f(\mathbf{a})+f(\mathbf{b})}{2\tau}\mathbf{\partial_c^\bot}
+
\frac{f(\mathbf{a})+f(\mathbf{c})}{2\tau}\mathbf{\partial_b^\bot}
+
\frac{f(\mathbf{b})+f(\mathbf{c})}{2\tau}\mathbf{\partial_a^\bot}
\end{equation}
where $\mathbf{\partial_c}=\mathbf{b}-\mathbf{a}$, $\mathbf{\partial_b}=\mathbf{a}-\mathbf{c}$, and $\mathbf{\partial_a}=\mathbf{c}-\mathbf{b}$, and $\tau$ is the 
area of the triangle $T$, as illustrated by the figure below
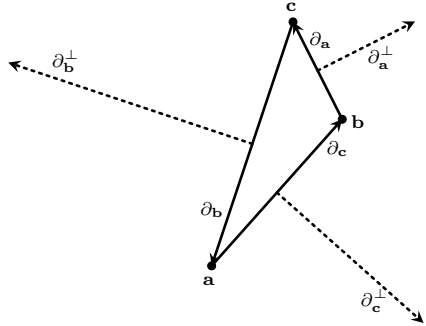
\begin{center}
\definecolor{ududff}{rgb}{0,0,0}
\begin{tikzpicture}[line cap=round,line join=round,>=triangle 45,x=0.9cm,y=0.9cm]
\clip(-3.8,-2.2) rectangle (2.8,2.73);
\draw [->,line width=1.pt,-stealth] (-0.6078038893645279,-1.232875438994985) -- (1.3211531808571966,0.9337067775584009);
\draw [->,line width=1.pt,-stealth] (1.3211531808571966,0.9337067775584007) -- (0.594299792077996,2.373435605332586);
\draw [->,line width=1.pt,-stealth] (0.594299792077996,2.373435605332586) -- (-0.6078038893645279,-1.2328754389949848);
\draw [->,line width=1.pt,dotted,-stealth] (-0.006752048643265951,0.5702800831688005) -- (-3.6130630929708367,1.7723837646113245);
\draw [->,line width=1.pt,dotted,-stealth] (0.9577264864675963,1.6535711914454934) -- (2.397455314241782,2.3804245802246937);
\draw [->,line width=1.pt,dotted,-stealth] (0.35667464574633434,-0.14958433071829214) -- (2.5232568622997205,-2.0785414009400163);
\begin{scriptsize}
\draw [fill=ududff] (-0.6078038893645279,-1.232875438994985) circle (1.5pt);
\draw[color=ududff] (-0.649737738717174,-1.45) node {$\mathbf{a}$};
\draw [fill=ududff] (0.594299792077996,2.373435605332586) circle (1.5pt);
\draw[color=ududff] (0.5663438925095653,2.6) node {$\mathbf{c}$};
\draw [fill=ududff] (1.3211531808571966,0.9337067775584007) circle (1.5pt);
\draw[color=ududff] (1.55,0.8987619030978622) node {$\mathbf{b}$};
\draw[color=black] (1.25,0.5) node {$\mathbf{\partial_c}$};
\draw[color=black] (1.,2.1) node {$\mathbf{\partial_a}$};
\draw[color=black] (-0.6,-0.4) node {$\mathbf{\partial_b}$};
\draw[color=black] (-2.7604081561336984,1.7933506892876472) node {$\mathbf{\partial_b^\bot}$};
\draw[color=black] (1.9,1.85) node {$\mathbf{\partial_a^\bot}$};
\draw[color=black] (1.8,-1.7430706061188466) node {$\mathbf{\partial_c^\bot}$};
\end{scriptsize}
\end{tikzpicture}
\end{center}

\begin{remark}
The vector quotient $\mathbf{q}$ is fully symmetric in the following sense:
\begin{center}
$
\mathbf{q}_{(f,\mathbf{v}_{\sigma(1)},\mathbf{v}_{\sigma(2)},\mathbf{v}_{\sigma(3)})}
=
\mathbf{q}_{(f,\mathbf{v}_1,\mathbf{v}_2,\mathbf{v}_3)}
$
for each $\sigma\in\mathcal{S}_3$,
\end{center}
where $\mathcal{S}_3$ is the symmetric group of all permutations of the set $\{1, 2, 3\}$.
\end{remark}

\begin{remark}
The Schwarz tangent paradox tells us that, despite the limit~(\ref{eq:strong der one-var}) exists when~$g\in C^1(I)$, 
\begin{center}
$\displaystyle
\lim_{
\arraycolsep=1.4pt\def\arraystretch{0.3}
\begin{array}{c}
\scriptstyle \mathbf{a},\mathbf{b},\mathbf{c}\to \mathbf{x}_0	\\
\scriptstyle \mathbf{a},\mathbf{b},\mathbf{c} \textrm{ not collinear}	
\end{array}
}
\kern-10pt
\mathbf{q}_{(f,\mathbf{a},\mathbf{b},\mathbf{c})}
\ \ \textrm{ doesn't exist,}
$
\end{center}
even when the non-linear two-variable function~$f$ is~$C^\infty(\mathbb{E}_2)$, as in the foregoing counterexample.
\end{remark}

\section{The geometric product of two vectors in $\mathbb{E}_2$.}

There are many ways to define the geometric product in an arbitrary quadratic space\footnote{See \cite{Riesz1993}, \cite{Porteous1981}, \cite{Choquet-Bruhat-DeWitt-Morette}, \cite{Guerlebeck-Sproessig}, \cite{Dorst}, \cite{Macdonald}, or \cite{Vaz-Rocha}, for instance.}. Here, for simplicity, we limit ourselves to present the geometric product between vectors of a two-dimensional Euclidean space $\mathbb{E}_2$, having $\mathbf{u}\cdot\mathbf{v}$ as positive definite symmetric bilinear form, and $|\mathbf{v}|=\sqrt{\mathbf{v}\cdot\mathbf{v}}$ as norm (where $\mathbf{u},\mathbf{v}$ are vectors of $\mathbb{E}_2$). 

\subsection{Axioms.}
we denote by $\mathbb{G}_2$ the associative algebra  of polynomials of vectors of $\mathbb{E}_2$ (poly-vectors, for short) satisfying the following rules: 

\begin{enumerate}
	\item[(A1)] 	scalars are considered as poly-vectors of degree $0$, and $\mathbf{v}^0=1$; non-zero vectors of $\mathbb{E}_2$ are considered as poly-vectors of degree 1;
	\item[(A4)] addition in $\mathbb{R}$, addition in $\mathbb{E}_2$, and addition in $\mathbb{G}_2$ they all coincide\footnote{This implies that the zero scalar coincides with the null vector in $\mathbb{G}_2$. That is, $0=\mathbf{0}$.}; multiplication in $\mathbb{R}$, multiplication of a scalar and a vector in $\mathbb{E}_2$, and multiplication of poly-vectors in $\mathbb{G}_2$ they all coincide\footnote{As a consequence of this axiom, it is natural to denote the geometric product by juxtaposition.};
	\item[(A3)]   scalars commute with vectors, that is 
\begin{center}
$\alpha \mathbf{v} = \mathbf{v} \alpha$
\end{center}
	  for each scalar $\alpha \in\mathbb{R}$, and vector $\mathbf{v}\in\mathbb{E}_2$;
	\item[(A5)] scalars and non-zero vectors of $\mathbb{E}_2$ are  linearly independent\footnote{That is, $\alpha+\beta\mathbf{v}=0$ if and only if $\alpha=\beta=0$ \big(where $\alpha,\beta\in\mathbb{R}$, and $\mathbf{0}\ne\mathbf{v}\in\mathbb{E}_2$\big).} in~$\mathbb{G}_2$.
	\item[(A6)] the Euclidean quadratic form~$|\mathbf{v}|^2=\mathbf{v}\cdot\mathbf{v}$ in $\mathbb{E}_2$ is the geometric product of the vector~$\mathbf{v}\in\mathbb{E}_2$ with itself, that is
\[
|\mathbf{v}|^2 = \mathbf{v}\mathbf{v} = \mathbf{v}^2\ .
\]
\end{enumerate}
Elements of $\mathbb{G}_2$ are called \textit{multivectors}. 

\subsection{Basic properties.}
A first consequence of the foregoing axioms concerns invertibility of non-zero vectors.
\begin{proposition*}
Every non-zero vector $\mathbf{v}\in\mathbb{E}_2$ is invertible with respect to the geometric product, and $\displaystyle \mathbf{v}^{-1} = \frac{1}{|\mathbf{v}|^2}\mathbf{v}$ is still a vector. 
\end{proposition*}
\begin{proof}
If $|\mathbf{v}|\ne 0$, then
$\displaystyle
\mathbf{v} \mathbf{v}^{-1} 
=
\mathbf{v} \frac{1}{|\mathbf{v}|^2} \mathbf{v} 
= 
\frac{1}{|\mathbf{v}|^2} \mathbf{v}\mathbf{v} 
= 1
$.
\end{proof}
\begin{remark}
By rule A3, notations like $\displaystyle \frac{\mathbf{v}}{|\mathbf{v}|}$ or $\displaystyle \frac{\mathbf{v}}{|\mathbf{v}|^2}$ are unambiguous. So, we can write~$\displaystyle \mathbf{v}^{-1} = \frac{\mathbf{v}}{|\mathbf{v}|^2}$.
\end{remark}
Axioms allow to clarify the nature of a particular symmetric poly-vector. 
\begin{proposition*}
The poly-vector $\displaystyle \frac{1}{2}(\mathbf{u}\mathbf{v}+\mathbf{v}\mathbf{u})$ is a number, whatever are the vectors $\mathbf{u}, \mathbf{v} \in\mathbb{E}_2$.
\end{proposition*} 
\begin{proof} By axiom A6, we have that $(\mathbf{u}+\mathbf{v})^2=|\mathbf{u}+\mathbf{v}|^2$ is a scalar. Moreover, 
\[
(\mathbf{u}+\mathbf{v})^2 
=
(\mathbf{u}+\mathbf{v})(\mathbf{u}+\mathbf{v}) 
=
\mathbf{u}\mathbf{u} + \mathbf{u}\mathbf{v} + \mathbf{v}\mathbf{u} + \mathbf{v}\mathbf{v} 
= 
|\mathbf{u}|^2 + \mathbf{u}\mathbf{v} + \mathbf{v}\mathbf{u} + |\mathbf{v}|^2\ .
\]
As also $|\mathbf{u}|^2$ and $|\mathbf{v}|^2$ are scalars, then we have that 
\begin{center}
$ 
\mathbf{u}\mathbf{v} + \mathbf{v}\mathbf{u} = (\mathbf{u}+\mathbf{v})^2  -|\mathbf{u}|^2 - |\mathbf{v}|^2 \in\mathbb{R}
$
\end{center}
\end{proof}
\begin{remarkn}\label{rmk:quasi-commutativity}
You can verify that the form $\beta$, corresponding to the foregoing poly-vector 
$
\beta(\mathbf{u},\mathbf{v})
= 
\frac{1}{2}(\mathbf{u}\mathbf{v}+\mathbf{v}\mathbf{u})
$,
is bilinear, symmetric and coincide with the Euclidean scalar product 
$\mathbf{u}\cdot\mathbf{v}$, that is 
$
\displaystyle \frac{1}{2}(\mathbf{u}\mathbf{v}+\mathbf{v}\mathbf{u})
=
\mathbf{u}\cdot\mathbf{v}
$, for each couple of vectors $\mathbf{u},\mathbf{v}\in\mathbb{E}_2$.
This allows to control the non-commutativity of the geometric product between vectors $\mathbf{u},\mathbf{v}$ in $\mathbb{E}_2$ as follows
\begin{center}
$
\mathbf{v}\mathbf{u}
=
2(\mathbf{u}\cdot\mathbf{v})-\mathbf{u}\mathbf{v}
$.
\end{center}
\end{remarkn}
As 
\begin{center}
$\displaystyle 
\mathbf{u}\mathbf{v} 
= 
\frac{1}{2}(\mathbf{u}\mathbf{v}+\mathbf{v}\mathbf{u}) 
+
\frac{1}{2}(\mathbf{u}\mathbf{v}-\mathbf{v}\mathbf{u}) 
=
(\mathbf{u}\cdot\mathbf{v}) 
+
\frac{1}{2}(\mathbf{u}\mathbf{v}-\mathbf{v}\mathbf{u})
$,
\end{center}
then geometric product will be completely explained once the nature of the poly-vector
\begin{center}
$\displaystyle
\mathbf{u}\wedge \mathbf{v} 
=
\frac{1}{2}(\mathbf{u}\mathbf{v}-\mathbf{v}\mathbf{u})
$
\end{center}
is clarified. 

We first notice that $\mathbf{u}\wedge\mathbf{u} = 0$, $\mathbf{u}\wedge\mathbf{v} =  -\mathbf{v}\wedge\mathbf{u}$, and $\mathbf{u}\wedge\mathbf{v}$ is bilinear, that is 
\begin{center}
$
(\alpha\mathbf{u})\wedge\mathbf{v} 
= 
\mathbf{u}\wedge(\alpha\mathbf{v}) 
=\alpha (\mathbf{u}\wedge\mathbf{v})
$, and
$\mathbf{u}\wedge(\mathbf{v}+\mathbf{w}) 
= 
(\mathbf{u}\wedge\mathbf{v}) + (\mathbf{u}\wedge\mathbf{w})
$
\end{center}
for every scalar $\alpha$, and vectors $\mathbf{u}, \mathbf{v}$, and $\mathbf{w}$ in $\mathbb{E}_2$.
 
In order to investigate the nature of the multivector $\mathbf{u}\wedge\mathbf{v}$,  let us observe that, if $\{\mathbf{g}_1,\mathbf{g}_2\}$ is an orthogonal basis for $\mathbb{E}_2$, $\mathbf{u} = \mu_1\mathbf{g}_1+ \mu_2\mathbf{g}_2$, and $\mathbf{v} = \nu_1\mathbf{g}_1+ \nu_2\mathbf{g}_2$, then
\begin{equation}
\label{eq:wedge1}
\mathbf{g}_1\cdot\mathbf{g}_2 = 0
\ \ ,\ \ 		
\mathbf{g}_1\wedge\mathbf{g}_2 
= 
\mathbf{g}_1\textbf{g}_2
\ \ ,\ \ 		
\textrm{ and }\ 
\mathbf{u}\wedge\mathbf{v} 
= 
\det
\left(
\begin{array}{cc}
\mu_1 & \mu_2 \\
\nu_1 & \nu_2 
\end{array}
\right)
\mathbf{g}_1\mathbf{g}_2
\end{equation}
We now prove that the dimension of the associative vector algebra $\mathbb{G}_2$ is $4$.
\begin{proposition}\label{prop:dimension}
If $\{\mathbf{g}_1,\mathbf{g}_2\}$ is an orthogonal basis of $\mathbb{E}_2$, then 
$\{1 ,\ \mathbf{g}_1 ,\ \mathbf{g}_2 ,\ \mathbf{g}_1\mathbf{g}_2\}$
is a basis of~$\mathbb{G}_2$.
\end{proposition}
\begin{proof}
Let us prove that 
$\{1\ ,\ \mathbf{g}_1\ ,\ \mathbf{g}_2\ ,\ \mathbf{g}_1\mathbf{g}_2\}$
generates $\mathbb{G}_2$.

Every element in $\mathbb{G}_2$ is a polynomial of vectors, that is a linear combination of product of vectors. As every vector is a linear combination of vectors $\mathbf{g}_1$, and $\mathbf{g}_2$, by the distributivity property, every product of vector can be written as a linear combination of product of vectors $\mathbf{g}_1$, and $\mathbf{g}_2$. As vectors $\mathbf{g}_1$, and $\mathbf{g}_2$ anti-commute ($\mathbf{g}_1\mathbf{g}_2=-\mathbf{g}_2\mathbf{g}_1$), every product of vectors $\mathbf{g}_1$, and $\mathbf{g}_2$ can be reduced to the product $\epsilon (\mathbf{g}_1)^h(\mathbf{g}_2)^k$ for some integer exponents $h,k\in\mathbb{N}$ and $\epsilon\in\{1,-1\}$. 
As $(\mathbf{g}_i)^l$ is a scalar if the integer exponent $l$ is even, and it is a multiple of vector $\mathbf{g}_i$ if $l$ is odd, then
every poly-vector can be rewritten as a linear combination of elements in $\{1\ ,\ \mathbf{g}_1\ ,\ \mathbf{g}_2\ ,\ \mathbf{g}_1\mathbf{g}_2\}$. This proves that $\{1\ ,\ \mathbf{g}_1\ ,\ \mathbf{g}_2\ ,\ \mathbf{g}_1\mathbf{g}_2\}$ generates $\mathbb{G}_2$. It remains to show that $1$, $\mathbf{g}_1$, $\mathbf{g}_2$, and $\mathbf{g}_1\mathbf{g}_2$ are linearly independent. This means that we have to show that, if 
\begin{equation}
\label{eq:lin indep}
\alpha_0 
+
\alpha_1\mathbf{g}_1 
+
\alpha_2\mathbf{g}_2 
+
\alpha_3 \mathbf{g}_1\mathbf{g}_2
= 0
\end{equation}
then necessarily $\alpha_0 = \alpha_1 = \alpha_2  = \alpha_3 = 0$.
First of all, let us observe that $\mathbf{g}_1\mathbf{g}_2\ne 0$. By contradiction, if $\mathbf{g}_1\mathbf{g}_2=0$, then by multiplying it from left by $(\mathbf{g}_1)^{-1}$ (recall that both vectors $\mathbf{g}_1$, and $\mathbf{g}_2$ are invertible), we would obtain $\mathbf{g}_2=0$, a contradiction with axiom A5.
Let us now multiply the expression~(\ref{eq:lin indep}) from left by $\mathbf{g}_1$, and from right by  $(\mathbf{g}_1)^{-1}$. Then, we can write  the following equivalent relations
\begin{center}
$
\mathbf{g}_1(\alpha_0 + \alpha_1\mathbf{g}_1 +\alpha_2\mathbf{g}_2 + \alpha_3 \mathbf{g}_1\mathbf{g}_2)(\mathbf{g}_1)^{-1}= 0 
$

$
\mathbf{g}_1\alpha_0(\mathbf{g}_1)^{-1}
+ \mathbf{g}_1\alpha_1\mathbf{g}_1(\mathbf{g}_1)^{-1}
+ \mathbf{g}_1\alpha_2\mathbf{g}_2(\mathbf{g}_1)^{-1}
+ \mathbf{g}_1\alpha_3 \mathbf{g}_1\mathbf{g}_2(\mathbf{g}_1)^{-1}= 0 
$

$
\alpha_0
+ \alpha_1\mathbf{g}_1
+ \alpha_2\mathbf{g}_1\mathbf{g}_2(\mathbf{g}_1)^{-1}
+ \alpha_3 \mathbf{g}_1\mathbf{g}_1\mathbf{g}_2(\mathbf{g}_1)^{-1}= 0 
$

$
\alpha_0
+ \alpha_1\mathbf{g}_1
- \alpha_2\mathbf{g}_2\mathbf{g}_1(\mathbf{g}_1)^{-1}
- \alpha_3 \mathbf{g}_1\mathbf{g}_2\mathbf{g}_1(\mathbf{g}_1)^{-1}= 0 
$

$
\alpha_0
+ \alpha_1\mathbf{g}_1
- \alpha_2\mathbf{g}_2
- \alpha_3 \mathbf{g}_1\mathbf{g}_2= 0 
$
\end{center}
By summing the last relation with~(\ref{eq:lin indep}), and then dividing the result by two,  we obtain that 
\begin{center}
$
\alpha_0
+ \alpha_1\mathbf{g}_1
= 0 
$
\end{center}
Rule A5 implies that $\alpha_0 = \alpha_1= 0$. So, the initial expression~(\ref{eq:lin indep}) reduces to
\begin{equation}
\label{eq:lin indep 2}
\alpha_2\mathbf{g}_2 
+
\alpha_3 \mathbf{g}_1\mathbf{g}_2
= 0
\end{equation}
Let us multiply the foregoing expression from left by $\mathbf{g}_2$, and from right by  $(\mathbf{g}_2)^{-1}$.
Then, we can write  the following equivalent relations 
\begin{center}
$
\mathbf{g}_2\alpha_2\mathbf{g}_2 (\mathbf{g}_2)^{-1}
+
\mathbf{g}_2\alpha_3 \mathbf{g}_1\mathbf{g}_2(\mathbf{g}_2)^{-1}
= 0
$

$
\alpha_2\mathbf{g}_2
+
\alpha_3 \mathbf{g}_2\mathbf{g}_1
= 0
$

$
\alpha_2\mathbf{g}_2
-
\alpha_3 \mathbf{g}_1 \mathbf{g}_2
= 0
$
\end{center}
As before, by summing the last relation with~(\ref{eq:lin indep 2}),  and then dividing the result by two,  we obtain that $\alpha_2=0$. So, we end to the expression $\alpha_3 \mathbf{g}\mathbf{g}_2= 0$. By multiplying it from left by $(\mathbf{g}_2)^{-1}(\mathbf{g}_1)^{-1}$, we obtain $\alpha_3=0$.

\end{proof}
The foregoing result and~(\ref{eq:wedge1}) imply that $\mathbf{u}\wedge\mathbf{v} = 0$ if and only if $\mathbf{u}$ and $\mathbf{v}$ are linearly dependent. 
We can also notice that the geometric product of two orthonormal vectors does not depend on those particular factors, but only on their order. 
\begin{proposition*}
If the basis $\{\mathbf{e}_1,\mathbf{e}_2\}$ for $\mathbb{E}_2$ is not only orthogonal, but also orthonormal \big(that is, $(\mathbf{e}_1)^2=(\mathbf{e}_2)^2=1$\big), then the geometric product $\mathbf{e}_1\mathbf{e}_2\in\mathbb{G}_2$ does not depend on the particular orthonormal basis $\{\mathbf{e}_1\ ,\ \mathbf{e}_2\}$, but only on its orientation (that is, on the order the two vectors of the basis are given). 
\end{proposition*}
\begin{proof}
If $\{\mathbf{e}'_1,\mathbf{e}'_2\}$ is any other orthonormal basis for $\mathbb{E}_2$, $\mathbf{e}'_1 = \epsilon_{1,1}\mathbf{e}_1+ \epsilon_{1,2}\mathbf{e}_2$ and $\mathbf{e}'_2 = \epsilon_{2,1}\mathbf{e}_1 + \epsilon_{2,2}\mathbf{e}_2$, then 
\begin{center}
$
\displaystyle 
\mathbf{e}'_1\mathbf{e}'_2 
= 
\det
\left(
\begin{array}{cc}
	\epsilon_{1,1} & \epsilon_{1,2}\\
	\epsilon_{2,1} & \epsilon_{2,2}
\end{array}
\right)
\mathbf{e}_1\mathbf{e}_2 
=
\pm \mathbf{e}_1\mathbf{e}_2
$
\end{center}
as 
$
\left(
\begin{array}{cc}
	\epsilon_{1,1} & \epsilon_{1,2}\\
	\epsilon_{2,1} & \epsilon_{2,2}
\end{array}
\right)
$ is an orthogonal matrix.
\end{proof}
That is why we can denote the multivector $\mathbf{e}_1\mathbf{e}_2=I_2$, and call it an \textit{orientation} of $\mathbb{E}_2$. 
So, by Proposition~\ref{prop:dimension}, the basic multivectors in $\mathbb{G}_2$ are scalars, vectors, and multiples of orientations.
You can also notice that
\begin{center}
$
(I_2)^2
= 
I_2 I_2 
=
\mathbf{e}_1\mathbf{e}_2 \mathbf{e}_1\mathbf{e}_2
=
\mathbf{e}_1(\mathbf{e}_2\mathbf{e}_1)\mathbf{e}_2 
=
\mathbf{e}_1(-\mathbf{e}_1\mathbf{e}_2)\mathbf{e}_2 
=
-\mathbf{e}_1\mathbf{e}_1\mathbf{e}_2\mathbf{e}_2
=
-(\mathbf{e}_1\mathbf{e}_1)(\mathbf{e}_2\mathbf{e}_2)
=
-1
$
\end{center}
So, we have that $I_2$ is invertible in $\mathbb{G}_2$ and $(I_2)^{-1}=-I_2$. Doesn't $I_2$ strongly resemble the imaginary unit ?

Finally, we can write the geometric product between two vectors $\mathbf{u},\mathbf{v}\in\mathbb{E}_2$ in terms of their coordinates with respect to any orthonormal basis $\{\mathbf{e}_1,\mathbf{e}_2\}$ of $\mathbb{E}_2$,
\begin{align*}
\mathbf{u}\mathbf{v}
& =
(\mathbf{u}\cdot\mathbf{v})
+ 
\mathbf{u}\wedge \mathbf{v} 
=
(\mu_1\nu_1 + \mu_2\nu_2) 
+
\det
\left(
\begin{array}{cc}
	\mu_1 & \mu_2\\
	\nu_1 & \nu_2
\end{array}
\right)
I_2 \\ 
& = |\mathbf{u}|\ |\mathbf{v}| (\cos \theta + I_2 \sin \theta)
\end{align*}
where $\mathbf{u} = \mu_1\mathbf{e}_1+ \mu_2\mathbf{e}-2$, and $\mathbf{v} = \nu_1\mathbf{e}_1+ \nu_2\mathbf{e}_2$, angle $\theta$ is oriented in $\mathbb{E}_2$ so that vector $\mathbf{u}$ can rotate towards vector $\mathbf{v}$ spanning angle $\theta \in (0, \pi)$, provided $\mathbf{u}$ and $\mathbf{v}$ are linearly independent (otherwise, $\mathbf{u}\mathbf{v} = \mathbf{u}\cdot\mathbf{v}$). 
 
\subsection{Comparing vectors in $\mathbb{E}_2 \subset \mathbb{G}_2$ with complex numbers.}
You have probably remarked some similarities with complex numbers ($I_2$ algebraically behaves in $\mathbb{G}_2$ as the imaginary unit $i$ does in $\mathbb{C}$). However, you should notice that while $\mathbb{G}_2$ has real-dimension four, $\mathbb{C}$ has real-dimension two. That is, scalars, vectors and orientations are distinguished in $\mathbb{G}_2$, while in $\mathbb{C}$ they all collapse into the notion of "complex number", viewed as a two-dimensional real vector. In other words, the geometric product of two vectors in $\mathbb{E}_2$ is not itself a vector: it is the sum of a scalar and a multiple of the orientation $I_2$. Instead, the product of two complex numbers \big(viewed as vectors in $\mathbb{R}^2$\big) is still a complex number \big(i.e., a vector in $\mathbb{R}^2$\big). 
Moreover, $\mathbb{C}$ is a division algebra (every non-zero element is invertible), while $\mathbb{G}_2$ is not: $1+ \mathbf{e}_1$, and $1- \mathbf{e}_1$ are not invertible in $\mathbb{G}_2$, as $(1+ \mathbf{e}_1)(1- \mathbf{e}_1)=0$. 
The product in $\mathbb{C}$ is commutative, while the geometric product in $\mathbb{G}_2$ is not.
The geometric product in $\mathbb{G}_2$ is invariant by rotations in $\mathbb{E}_2$, while the product in $\mathbb{C}$ is not invariant by rotations in $\mathbb{R}^2$.

\subsection{The determinant of a $2\times 2$ real matrix viewed in $\mathbb{G}_2$.}
Let us notice that, for every vector $\mathbf{v}\in\mathbb{E}_2$, then
\begin{align*}
\mathbf{v} I_2 
& =
(\nu_1\mathbf{e}_1+ \nu_2\mathbf{e}_2)\mathbf{e}_1\mathbf{e}_2 
=
\nu_1\mathbf{e}_1\mathbf{e}_1\mathbf{e}_2 + \nu_2\mathbf{e}_2\mathbf{e}_1\mathbf{e}_2 
=
-
\nu_1\mathbf{e}_1\mathbf{e}_2\mathbf{e}_1 
-
\nu_2\mathbf{e}_1\mathbf{e}_2\mathbf{e}_2 \\
& =
-\mathbf{e}_1\mathbf{e}_2(\nu_1\mathbf{e}_1 + \nu_2\mathbf{e}_2) 
=
-I_2 \mathbf{v} = \nu_1\mathbf{e}_2- \nu_2\mathbf{e}_1\in\mathbb{E}_2
\end{align*}
which corresponds to the vector obtained by rotating vector $\mathbf{v}$ of a right angle according to the orientation $I_2=\mathbf{e}_1\mathbf{e}_2$, as illustrated below,
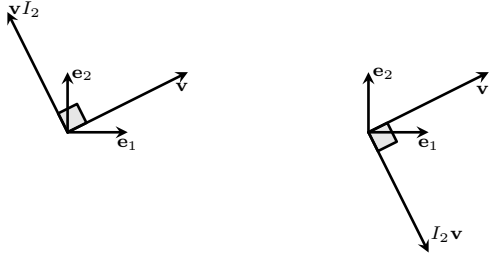
\definecolor{qqwuqq}{rgb}{0,0,0}
\begin{center}
\begin{tikzpicture}[line cap=round,line join=round,x=0.8cm,y=0.8cm]
\clip(-1.1343801652892544,-2.1) rectangle (7.2,2.2);
\draw[line width=1.pt,color=qqwuqq,fill=qqwuqq,fill opacity=0.10000000149011612] (0.31361431340512835,0.15680715670256418) -- (0.1568071567025642,0.4704214701076926) -- (-0.15680715670256418,0.31361431340512835) -- (0.,0.) -- cycle; 
\draw[line width=1.pt,color=qqwuqq,fill=qqwuqq,fill opacity=0.10000000149011612] (5.156807156702564,-0.31361431340512835) -- (5.470421470107692,-0.15680715670256418) -- (5.313614313405129,0.1568071567025642) -- (5.,0.) -- cycle; 
\draw [->,line width=1.pt,-stealth] (0.,0.) -- (1.,0.);
\draw [->,line width=1.pt,-stealth] (0.,0.) -- (0.,1.);
\draw [->,line width=1.pt,-stealth] (0.,0.) -- (2.,1.);
\draw [->,line width=1.pt,-stealth] (0.,0.) -- (-1.,2.);
\draw [->,line width=1.pt,-stealth] (5.,0.) -- (6.,0.);
\draw [->,line width=1.pt,-stealth] (5.,0.) -- (5.,1.);
\draw [->,line width=1.pt,-stealth] (5.,0.) -- (7.,1.);
\draw [->,line width=1.pt,-stealth] (5.,0.) -- (6.,-2.);
\begin{scriptsize}
\draw[color=black] (1,-0.2) node {$\mathbf{e}_1$};
\draw[color=black] (0.25,0.9555371900826521) node {$\mathbf{e}_2$};
\draw[color=black] (1.9,0.7489256198347183) node {$\mathbf{v}$};
\draw[color=black] (-0.7128925619834716,2.046446280991743) node {$\mathbf{v}I_2$};
\draw[color=black] (6,-0.2) node {$\mathbf{e}_1$};
\draw[color=black] (5.25,0.9803305785124042) node {$\mathbf{e}_2$};
\draw[color=black] (6.9,0.7) node {$\mathbf{v}$};
\draw[color=black] (6.3,-1.7) node {$I_2\mathbf{v}$};
\end{scriptsize}
\end{tikzpicture}
\end{center}
Thanks to the foregoing results we can now proceed to show the two key facts that link the geometric product to the basic notion of secant plane.
\begin{proposition*}
The determinant of a 2x2 real matrix is a Clifford ratio in $\mathbb{G}_2$.
\end{proposition*}
\begin{proof}
Let us consider the rows of a $2\times 2$ real matrix
\[
\left(
\begin{array}{cc}
		\mu_1 & \mu_2 \\
		\nu_1 & \nu_2 
\end{array}
\right)
\]
as the components of two vectors $\mathbf{u} = \mu_1\mathbf{e}_1+ \mu_2\mathbf{e}_2$, and $\mathbf{v} = \nu_1\mathbf{e}_1+ \nu_2\mathbf{e}_2$ in $\mathbb{E}_2$ with respect to some orthonormal basis $\{\mathbf{e}_1,\mathbf{e}_2\}$. 

As 
$
\mathbf{u}\wedge \mathbf{v} 
=
\det
\left(
\begin{array}{cc}
		\mu_1 & \mu_2 \\
		\nu_1 & \nu_2 
\end{array}\right)
\mathbf{e}_1\mathbf{e}_2 
=
\det
\left(
\begin{array}{cc}
		\mu_1 & \mu_2 \\
		\nu_1 & \nu_2 
\end{array}
\right)
I_2
$, then we can write
\[
\det
\left(
\begin{array}{cc}
		\mu_1 & \mu_2 \\
		\nu_1 & \nu_2 
\end{array}
\right)
=
(\mathbf{u}\wedge \mathbf{v})
(I_2)^{-1}
\ , \
\textrm{which is a ratio in } \mathbb{G}_2\ .
\]
\end{proof}
\begin{remark}
You can verify that $(\mathbf{u}\wedge \mathbf{v})(I_2)^{-1} = (I_2)^{-1}(\mathbf{u}\wedge \mathbf{v})$. So the expression  
$
\displaystyle
\det
\left(
\begin{array}{cc}
		\mu_1 & \mu_2 \\
		\nu_1 & \nu_2 
\end{array}
\right)
=
\frac{
\mathbf{u}\wedge \mathbf{v}}
{I_2}
$
for the determinant of a $2\times 2$ matrix is unambiguous. 
\end{remark}

\begin{proposition}\label{prop:det=scalar prod}
The determinant a 2x2 real matrix is a scalar product. 
\end{proposition}
\begin{proof} By using the same assumptions of the foregoing proof, we have that
\begin{align*}
\det
\left(
\begin{array}{cc}
		\mu_1 & \mu_2 \\
		\nu_1 & \nu_2 
\end{array}
\right)
& =
(\mathbf{u}\wedge \mathbf{v})
(I_2)^{-1}
= 
-(\mathbf{u}\wedge \mathbf{v})I_2
=
(\mathbf{v}\wedge \mathbf{u})I_2
=
\frac{1}{2}(\mathbf{v}\mathbf{u}-\mathbf{u}\mathbf{v})I_2\\
& =
\frac{1}{2}(\mathbf{v}\mathbf{u}I_2-\mathbf{u}\mathbf{v}I_2)
=
\frac{1}{2}\big[\mathbf{v}(\mathbf{u}I_2)+(\mathbf{u}I_2)\mathbf{u}\big]
=
(\mathbf{u} I_2)\cdot\mathbf{v}\ .
\end{align*}
\end{proof}

\section{\kern-10pt The difference vector quotient of a secant plane.}
Let us reformulate the relations defining the plane secant the graph of a function.
In the one-variable case the equation of the line secant the graph of $g:I\to\mathbb{R}$ at points $\big(a,g(a)\big), \big(b,g(b)\big)\in I\times\mathbb{R}$ can be written, in the Cartesian $(x,z)$-plane, in two equivalent ways: (\ref{eq:secant one-var}) or
\[
\det
\left(
\begin{array}{cc}
	x-a & z-g(a)\\
	b-a	& g(b)-g(a)
\end{array}
\right)
=0\ .
\]
In the $(x,y,z)$-coordinate system, the equation of the plane secant the graph of $f$ at points $(\alpha_1,\alpha_2,f(\mathbf{a})), (\beta_1,\beta_2, f(\mathbf{b})), (\gamma_1,\gamma_2,f(\mathbf{c}))\in\mathbb{R}^3$ can be written as
\begin{equation}
\label{eq:equiv two dim coord}
\det
\left(
\begin{array}{ccc}
	x-\alpha_1 & y-\alpha_2 & z-f(\mathbf{a})\\
	\beta_1-\alpha_1	& \beta_2-\alpha_2 & f(\mathbf{b})-f(\mathbf{a}) \\
	\gamma_1-\alpha_1	& \gamma_2-\alpha_2 & f(\mathbf{c})-f(\mathbf{a}) 
\end{array}
\right)
=0
\end{equation}
where 
$\mathbf{a}=\alpha_1\mathbf{e}_1+\alpha_2\mathbf{e}_2$,
$\mathbf{b}=\beta_1\mathbf{e}_1+\beta_2\mathbf{e}_2$, 
$\mathbf{c}=\gamma_1\mathbf{e}_1+\gamma_2\mathbf{e}_2$ are vectors (that we also call points) in the two-dimensional domain $\Omega\subseteq\mathbb{E}_2$ of $f$, and $\{\mathbf{e}_1,\mathbf{e}_2\}$ is an orthonormal basis of $\mathbb{E}_2$. In the following, we show how to rewrite~(\ref{eq:equiv two dim coord}) in the $\mathbb{E}_2$-coordinate-free setting~(\ref{eq:equiv two dim}).
The following equivalences start from a Laplace expansion of the determinant~(\ref{eq:equiv two dim coord}). 
\begin{center}
$
\big(z-f(\mathbf{a}) \big) 
{\scriptsize
\det
\left(
\begin{array}{cc}
	\beta_1-\alpha_1 & \beta_2-\alpha_2\\
	\gamma_1-\alpha_1 & \gamma_2-\alpha_2
\end{array}
\right)} 
-\big(f(\mathbf{b})-f(\mathbf{a})\big) 
{\scriptsize
\det
\left(
\begin{array}{cc}
	x-\alpha_1 & y-\alpha_2\\
	\gamma_1-\alpha_1 & \gamma_2-\alpha_2
\end{array}
\right)}
+
+
\big(f(\mathbf{c})-f(\mathbf{a})\big)
{\scriptsize
\det
\left(
\begin{array}{cc}
	x-\alpha_1 & y-\alpha_2\\
	\beta_1-\alpha_1 & \beta_2-\alpha_2
\end{array}
\right)}
= 0 
$
\end{center}

\begin{center}
$
\big(z-f(\mathbf{a})\big) 
\big[(\mathbf{b}-\mathbf{a})\wedge(\mathbf{c}-\mathbf{a})\big] (I_2)^{-1} 
=
\big(f(\mathbf{b})
\kern-2pt-\kern-2pt
f(\mathbf{a})\big) 
\big[(\mathbf{x}-\mathbf{a})
\kern-2pt\wedge\kern-2pt
(\mathbf{c}-\mathbf{a})\big](I_2)^{-1}
-
\big(f(\mathbf{c})
\kern-2pt-\kern-2pt
f(\mathbf{a})\big)
\big[(\mathbf{x}-\mathbf{a})
\kern-2pt\wedge\kern-2pt
(\mathbf{b}-\mathbf{a})\big](I_2)^{-1}
$
\end{center}

\begin{center}
$
\big(z-f(\mathbf{a})\big)
\big[(\mathbf{b}-\mathbf{a})\wedge(\mathbf{c}-\mathbf{a})\big] 
=
\big(f(\mathbf{b})-f(\mathbf{a})\big)
\big[(\mathbf{x}-\mathbf{a})\wedge(\mathbf{c}-\mathbf{a})\big] 
-
\big(f(\mathbf{c})-f(\mathbf{a})\big)
\big[(\mathbf{x}-\mathbf{a})\wedge(\mathbf{b}-\mathbf{a})\big]
$
\end{center}

\begin{center}
$ 
z-f(\mathbf{a})
=
\big(f(\mathbf{b})-f(\mathbf{a})\big)
\big[(\mathbf{x}-\mathbf{a})\wedge(\mathbf{c}-\mathbf{a})\big]
\big[(\mathbf{b}-\mathbf{a})\wedge(\mathbf{c}-\mathbf{a})\big]^{-1}
+ -
\big(f(\mathbf{c})-f(\mathbf{a})\big)
\big[(\mathbf{x}-\mathbf{a})\wedge(\mathbf{b}-\mathbf{a})\big]
\big[(\mathbf{b}-\mathbf{a})\wedge(\mathbf{c}-\mathbf{a})\big]^{-1} 
$
\end{center}

\begin{center}
$
z 
=
f(\mathbf{a}) 
+
\big(f(\mathbf{b})-f(\mathbf{a})\big)
\big[(\mathbf{x}-\mathbf{a})\wedge(\mathbf{c}-\mathbf{a})\big]
\big[(\mathbf{b}-\mathbf{a})\wedge(\mathbf{c}-\mathbf{a})\big]^{-1}
+ -
\big(f(\mathbf{c})-f(\mathbf{a})\big)
\big[(\mathbf{x}-\mathbf{a})\wedge(\mathbf{b}-\mathbf{a})\big]
\big[(\mathbf{b}-\mathbf{a})\wedge(\mathbf{c}-\mathbf{a})\big]^{-1} 
$
\end{center}
\begin{align*}
\textrm{As }\ (\mathbf{b}-\mathbf{a})\wedge(\mathbf{c}-\mathbf{a})
& =
\mathbf{a}\wedge\mathbf{b}
+
\mathbf{b}\wedge\mathbf{c}
+
\mathbf{c}\wedge\mathbf{a}\\
& =
I_2 
\underbrace{
\left[
\scriptstyle
\det
\left(
\begin{array}{cc}
	\alpha_1 & \alpha_2\\
	\beta_1 & \beta_2
\end{array}
\right)
+
\det
\left(
\begin{array}{cc}
		\beta_1 & \beta_2\\
		\gamma_1 & \gamma_2
\end{array}
\right)
+
\det
\left(
\begin{array}{cc}
	\gamma_1 & \gamma_2\\
	\alpha_1 & \alpha_2
\end{array}
\right)
\right]}_{2\tau}\ ,
\end{align*}
where $\tau$ is simply the oriented\footnote{The sign of $\tau$ is positive if and only if the geometric ratio between $(\mathbf{b}-\mathbf{a})\wedge(\mathbf{c}-\mathbf{a})$ and $I_2$ is positive. See also~\cite{Braden86}.} area of the triangle having as vertices the points $\mathbf{a}$, $\mathbf{b}$, and $\mathbf{c}$), the foregoing equivalences, describing the secant plane, can continue as follows
\begin{center}
$
z 
=
f(\mathbf{a}) 
+
\frac{f(\mathbf{b})-f(\mathbf{a})}{2\tau}
\big[(\mathbf{x}-\mathbf{a})\wedge(\mathbf{c}-\mathbf{a})\big]
I_2^{-1}
-
\frac{f(\mathbf{c})-f(\mathbf{a})}{2\tau}
\big[(\mathbf{x}-\mathbf{a})\wedge(\mathbf{b}-\mathbf{a})\big]
I_2^{-1} 
$
\end{center}

\begin{center}
$
z 
=
f(\mathbf{a}) 
-
\frac{f(\mathbf{b})-f(\mathbf{a})}{2\tau}
\big[(\mathbf{c}-\mathbf{a})\wedge(\mathbf{x}-\mathbf{a})\big]
I_2^{-1}
+
\frac{f(\mathbf{c})-f(\mathbf{a})}{2\tau}
\big[(\mathbf{b}-\mathbf{a})\wedge(\mathbf{x}-\mathbf{a})\big]
I_2^{-1} 
$
\end{center}
By proposition~\ref{prop:det=scalar prod}, we can write
\begin{center}
$
z 
=
f(\mathbf{a}) 
-
\frac{f(\mathbf{b})-f(\mathbf{a})}{2\tau}
\Big\{\kern-2pt
\big[(\mathbf{c}-\mathbf{a})I_2\big] \cdot(\mathbf{x}-\mathbf{a})
\kern-2pt\Big\}
+
\frac{f(\mathbf{c})-f(\mathbf{a})}{2\tau}
\Big\{\kern-2pt
\big[(\mathbf{b}-\mathbf{a})I_2\big]\cdot(\mathbf{x}-\mathbf{a})
\kern-2pt\Big\}
$
\end{center}

\begin{center}
$
z 
=
f(\mathbf{a}) 
-
\Big\{\kern-2pt
\frac{f(\mathbf{b})-f(\mathbf{a})}{2\tau}
\big[(\mathbf{c}-\mathbf{a})I_2\big] 
-
\frac{f(\mathbf{c})-f(\mathbf{a})}{2\tau}
\big[(\mathbf{b}-\mathbf{a})I_2
\big]
\kern-2pt\Big\}
\cdot(\mathbf{x}-\mathbf{a})
$
\end{center}

\begin{center}
$
z 
=
f(\mathbf{a}) 
-
\frac{1}{2\tau}
\Big\{\kern-2pt
\Big[
\big(f(\mathbf{b})-f(\mathbf{a})\big)
(\mathbf{c}-\mathbf{a}) 
-
\big(f(\mathbf{c})-f(\mathbf{a})\big)
(\mathbf{b}-\mathbf{a})
\Big]I_2
\kern-2pt\Big\}
\cdot(\mathbf{x}-\mathbf{a})
$
\end{center}
This allows to explicitly write vector 
$
\mathbf{q}
=
\mathbf{q}_{(f,\mathbf{a},\mathbf{b},\mathbf{c})}
$
of expression~(\ref{eq:equiv two dim}) 
\begin{align*}
\mathbf{q}
& =
-
\frac{1}{2\tau}
\Big\{\kern-2pt
\Big[
\big(f(\mathbf{b})-f(\mathbf{a})\big)
(\mathbf{c}-\mathbf{a}) 
-
\big(f(\mathbf{c})-f(\mathbf{a})\big)
(\mathbf{b}-\mathbf{a})
\Big]I_2\\
&=
\Big[
\big(f(\mathbf{b})-f(\mathbf{a})\big)
(\mathbf{c}-\mathbf{a}) 
-
\big(f(\mathbf{c})-f(\mathbf{a})\big)
(\mathbf{b}-\mathbf{a})
\Big]
\big[(\mathbf{b}-\mathbf{a})\wedge(\mathbf{c}-\mathbf{a})\big]^{-1} \ ,
\end{align*}
and proves the following main result of this article.
\begin{theorem}\label{thm:main}
The plane secant the graph of a two-variable function $f:\Omega\subseteq\mathbb{E}_2\to\mathbb{R}$ at points 
$\big(\mathbf{a},f(\mathbf{a})\big)$, $\big(\mathbf{b},f(\mathbf{b})\big)$, and $\big(\mathbf{c},f(\mathbf{c})\big)$ \big(where $\mathbf{a}$, $\mathbf{b}$, and $\mathbf{c}$ are three non-collinear points in the domain $\Omega$ of $f$\big), can be represented in the $(\mathbf{v},z)$-space $\mathbb{E}_2\times\mathbb{R}$ by the relation
\[
z=f(\mathbf{a}) + \mathbf{q} \cdot (\mathbf{v}-\mathbf{a})\ ,
\]
where vector~$\mathbf{q}=\mathbf{q}_{(f,\mathbf{a},\mathbf{b},\mathbf{c})}$ is the geometric quotient in $\mathbb{G}_2$ between the vector
\begin{center}
$
\big(f(\mathbf{b})-f(\mathbf{a})\big)
(\mathbf{c}-\mathbf{a}) 
-
\big(f(\mathbf{c})-f(\mathbf{a})\big)
(\mathbf{b}-\mathbf{a})
$,
\end{center}
and the orientation
\begin{center}
$
(\mathbf{b}-\mathbf{a})\wedge(\mathbf{c}-\mathbf{a})
$.
\end{center}
\end{theorem}

\section{The vector quotient as a linear combination in~$\mathbb{E}_2$.}
Computations in $\mathbb{G}_2$ allow to explicitly express the foregoing difference vector quotient~$\mathbf{q}=\mathbf{q}_{(f,\mathbf{a},\mathbf{b},\mathbf{c})}$ as a linear combination of normals to the  triangle defined by $\mathbf{a}$, $\mathbf{b}$, and $\mathbf{c}$.

Let $
\mathbf{\partial_a}
=
\mathbf{c}-\mathbf{b}
$,
\ 
$
\mathbf{\partial_b}
=
\mathbf{a}-\mathbf{c}
$,
\ 
$
\mathbf{\partial_c}
=
\mathbf{b}-\mathbf{a}
$; so,
$
\mathbf{\partial_a}+\mathbf{\partial_b}+\mathbf{\partial_c}=0
$. 
Moreover 
\begin{center}
$
(\mathbf{b}-\mathbf{a})\wedge(\mathbf{c}-\mathbf{a})
=
\mathbf{\partial_a}\wedge\mathbf{\partial_b}
=
\mathbf{\partial_b}\wedge\mathbf{\partial_c}
=
\mathbf{\partial_c}\wedge\mathbf{\partial_a}
=
2\tau I_2 
$,
\end{center}
where $\tau$ is the signed area of the oriented triangle determined by the ordered points $\mathbf{a}$, $\mathbf{b}$, and~$\mathbf{c}$, whose sign depends on the orientation~$I_2$. Then,
\begin{align*}
&
\big(f(\mathbf{b})
\kern-3pt-\kern-3pt
f(\mathbf{a})\big)
(\mathbf{c}
\kern-3pt-\kern-3pt
\mathbf{a}) 
\kern-2pt-\kern-2pt
\big(f(\mathbf{c})
\kern-3pt-\kern-3pt
f(\mathbf{a})\big)
(\mathbf{b}
\kern-3pt-\kern-3pt
\mathbf{a})
=
\big(f(\mathbf{a})
\kern-3pt-\kern-3pt
f(\mathbf{b})\big)
\mathbf{\partial_b}
\kern-2pt-\kern-2pt
\big(f(\mathbf{c})
\kern-3pt-\kern-3pt
f(\mathbf{a})\big)
\mathbf{\partial_c}
\\
= & 
f(\mathbf{a})
\big[
\mathbf{\partial_b}
+
\mathbf{\partial_c}
\big]
- f(\mathbf{b})\mathbf{\partial_b}
- f(\mathbf{c})\mathbf{\partial_c}
=
-\big[
f(\mathbf{a})\mathbf{\partial_a}
+
f(\mathbf{b})\mathbf{\partial_b}
+
f(\mathbf{c})\mathbf{\partial_c}
\big]\\
= &
f(\mathbf{a})
\big(\mathbf{\partial_b}+\mathbf{\partial_c}\big)
+
f(\mathbf{b})
\big(\mathbf{\partial_a}+\mathbf{\partial_c}\big)
+
f(\mathbf{c})
\big(\mathbf{\partial_a}+\mathbf{\partial_b}\big)\\
= &
\big[f(\mathbf{a})+f(\mathbf{b})\big]
\mathbf{\partial_c}
+
\big[f(\mathbf{a})+f(\mathbf{c})\big]
\mathbf{\partial_b}
+
\big[f(\mathbf{b})+f(\mathbf{cb})\big]
\mathbf{\partial_a}
\end{align*}
So, we can write
\begin{align*}
\mathbf{q}
& =
\left\{
\big[f(\mathbf{a})+f(\mathbf{b})\big]
\mathbf{\partial_c}
+
\big[f(\mathbf{a})+f(\mathbf{c})\big]
\mathbf{\partial_b}
+
\big[f(\mathbf{b})+f(\mathbf{c})\big]
\mathbf{\partial_a}
\right\}
\big(2\tau I_2\big)^{-1}\\
& =
\frac{f(\mathbf{a})+f(\mathbf{b})}{2\tau}
\mathbf{\partial_c}(I_2)^{-1}
+
\frac{f(\mathbf{a})+f(\mathbf{c})}{2\tau}
\mathbf{\partial_b}(I_2)^{-1}
+
\frac{f(\mathbf{b})+f(\mathbf{c})}{2\tau}
\mathbf{\partial_a}(I_2)^{-1}\\
& =
\frac{f(\mathbf{a})+f(\mathbf{b})}{2\tau}
\underbrace{I_2\mathbf{\partial_c}}_{\mathbf{\partial^\bot_c}}
+
\frac{f(\mathbf{a})+f(\mathbf{c})}{2\tau}
\underbrace{I_2\mathbf{\partial_b}}_{\mathbf{\partial^\bot_b}}
+
\frac{f(\mathbf{b})+f(\mathbf{c})}{2\tau}
\underbrace{I_2\mathbf{\partial_a}}_{\mathbf{\partial^\bot_a}}\ .
\end{align*}
This proves our second theorem, which generalizes a lemma proved\footnote{Only for linear functions.} in~\cite{Daws2013}.
\begin{theorem}\label{thm:lin comb orth}
The plane secant the graph of a two-variable function $f:\Omega\subseteq\mathbb{E}_2\to\mathbb{R}$ at points 
$\big(\mathbf{a},f(\mathbf{a})\big)$, $\big(\mathbf{b},f(\mathbf{b})\big)$, and $\big(\mathbf{c},f(\mathbf{c})\big)$ \big(where $\mathbf{a}$, $\mathbf{b}$, and $\mathbf{c}$ are three non-collinear points in the domain $\Omega$ of $f$\big), can be represented in the $(\mathbf{v},z)$-space $\mathbb{E}_2\times\mathbb{R}$ by the relation 
$
z=f(\mathbf{a}) + \mathbf{q} \cdot (\mathbf{v}-\mathbf{a})
$, 
where vector~$\mathbf{q}=\mathbf{q}_{(f,\mathbf{a},\mathbf{b},\mathbf{c})}$ is the linear combination~(\ref{eq:lin comb orth}). 
\end{theorem}
\begin{remark}
If $f$ is the affine function $f(\mathbf{x})=(\mathbf{v}\cdot \mathbf{x})+\phi$, for some $\mathbf{v}\in\mathbb{E}_2$ and $\phi\in\mathbb{R}$, then you can verify that $\mathbf{q}_{(f,\mathbf{a},\mathbf{b},\mathbf{c})}=\mathbf{v}$.
\end{remark}


\vfill\eject

\end{document}